\documentclass[12pt, reqno]{amsart}

\usepackage{hyperref}
\usepackage{enumitem}

\usepackage{xcolor}

\newtheorem{theorem}{Theorem}
\newtheorem{prop}[theorem]{Proposition}
\newtheorem{corollary}[theorem]{Corollary}
\newtheorem{lemma}[theorem]{Lemma}
\newtheorem{conjecture}[theorem]{Conjecture}

\theoremstyle{definition}

\newtheorem{remark}[theorem]{Remark}

\begin{document}

\title{Hurwitz orbits of equal size}

\author{Colin Pirillo and Seth Sabar}

\maketitle

\begin{abstract}
We provide a variety of cases in which two factorizations have Hurwitz orbits of the same size. We begin with prototypical results about factorizations of length two, and show that cycling elements or flipping and inverting elements in any factorization preserves Hurwitz orbit size. We prove that ``double reverse" factorizations in groups with special presentations have Hurwitz orbits of equal size, and end with applications to complex reflection groups.
\end{abstract}

\section{Introduction}

The Hurwitz action is a combinatorial action on tuples of group elements, based on basic moves that swap two elements, modifying one, while preserving the product. Its study dates to work of Hurwitz in the late 19th century. This paper is motivated by an observation about the Hurwitz action and the complex reflection group $G_6 = \langle a, b \mid a^3 = b^2 = 1, ababab = bababa \rangle$. We took factorizations containing only $a$, $b$, and $a^{-1}$, and randomly permuted the factors. Permuting elements in these factorizations preserved the size of the their Hurwitz orbits (see Conjecture \ref{conjecture:g6}). This observation raised the question ``In what ways can we reorder or modify the factors in a factorization while preserving the size of its Hurwitz orbit?"

In Section \ref{section:length 2}, we prove that for elements $x$ and $y$ in a group $G$, the Hurwtiz orbits of $(x,y)$, $(x^{-1},y^{-1})$, and $(y,x)$ have equal size. We then generalize these results to longer factorizations in several ways. We prove in Section \ref{section:cycling} that for any $x_1, x_2, \ldots ,x_\ell$ in a group $G$, $(x_1,x_2, \ldots, x_\ell)$ and $(x_2, \ldots, x_\ell, x_1)$ have Hurwitz orbits of equal size. In Section \ref{reverses} we show that $(x_\ell^{-1} \ldots,x_2^{-1}, x_1^{-1})$ belongs to an orbit of the same size as well. In the special case where all $x_1, x_2, \ldots, x_\ell$ have order $1$ or $2$, we see that the Hurwitz orbit of $(x_1,x_2, \ldots, x_\ell)$ has the same size as the orbit of $(x_\ell, x_{\ell-1}, \ldots x_1)$.

In Section \ref{Section:ReverseRelation}, we prove that $(x_1,x_2, \ldots, x_\ell)$ and $(x_\ell, x_{\ell-1}, \ldots x_1)$ have Hurwitz orbits of equal size when the $x_i$ belong to a generating set in a group presentation with ``reversible relations". We end in Section \ref{Section:Applications} with applications to Coxeter groups and Shephard groups, as well as a conjecture about $G_6$.

\subsection*{Acknowledgements}

The research presented was conducted at The George Washington University while we were students at School Without Walls High School, located in Foggy Bottom, Washington DC. We would like to thank Gaurav Gawankar, Dounia Lazreq, and Mehr Rai for helpful conversations which contributed to our understanding of this material. We would especially like to thank Professor Joel Brewster Lewis for his mentorship throughout the research process.

\section{Background}
\label{sec:background}

In a group $G$, a \emph{factorization} of an element $g \in G$ is a tuple of elements in $G$ that multiply to $g$. Given a factorization $(x_1, x_2, \ldots, x_\ell)$ of an element in $G$, a \emph{Hurwitz move} at position $i$, with $1 \leq i \leq \ell -1$, is defined to be the following operation:
\[
\begin{array}{cccl}
(x_1, \ldots, x_{i-1}, &  x_i, & x_{i+1}, & x_{i+2}, \ldots, x_\ell) \overset{\sigma_{i}}{\to} \\
(x_1, \ldots, x_{i-1}, &  x_{i+1}, & x_{i+1}^{-1}x_{i}x_{i+1}, & x_{i+2}, \ldots, x_\ell).
\end{array}
\]
The inverse $\sigma_i^{-1}$ of the Hurwitz move $\sigma_i$ is given by
\[
\begin{array}{cccl}
(x_1, \ldots, x_{i-1}, & x_i, & x_{i+1}, & x_{i+2}, \ldots, x_\ell) \overset{\sigma_{i}^{-1}}{\to} \\
(x_1, \ldots, x_{i-1}, & x_{i}x_{i+1}x_{i}^{-1}, & x_{i}, & x_{i+2}, \ldots, x_\ell).
\end{array}
\]
The \emph{Hurwitz orbit} of a factorization $T$ is the set of factorizations which can be produced by applying Hurwitz moves to $T$. We use the term ``orbit'' because Hurwitz moves give rise to an action of the braid group. The term \emph{size} refers to the cardinality of the Hurwitz orbit.

The first paper to study the Hurwitz action was written by A. Hurwitz in 1891 \cite{Hur1891}. Historically, researchers who have studied the Hurwitz action have been interested in when two factorizations belong to the same Hurwitz orbit. For example, \cite{Kluitmann} and \cite{Ben-Itzhak} answered the question of when two transposition factorizations in the symmetric group belong to the same Hurwitz orbit. \cite{Hou}, \cite{Sia}, and \cite{Berger} answered this same question, instead for factorizations in the dihedral group. Many authors have also considered the case of factorizations of special elements in complex reflection groups or Coxeter groups \cite{Bessis, Lewis, Lewis2, Baum, Igusa, Peterson, GLRS}. Researchers have explored the Hurwitz action from different perspectives as well. In 2019, Muhle and Ripoll studied the Hurwitz action as a connectivity property on posets \cite{Muhle}. Unlike previous papers, our paper is concerned not with the question of when two factorizations belong to the same Hurwitz orbit, but instead when they have Hurwitz orbits of the same size.

\section{Factorizations of length two}
\label{section:length 2}
In this section we prove that $(x,y)$, $(y,x)$, and $(x^{-1},y^{-1})$ have Hurwitz orbits of the same size. This observation on the orbit sizes of length-$2$ factorizations serves as a prototype and motivating example for the more general results in later sections. We begin with a lemma on the general form of $(x,y)$ after any number of Hurwitz moves. 

\begin{lemma}
\label{lemma:xy-form}

Let $G$ be a group, $n$ an integer, and $x,y \in G$. Then 
\begin{enumerate}[label={\rm (\alph*)}]
\item $\sigma^{2n}(x,y) = \left(y^{-1}(x^{-1}y^{-1})^{n-1}x(yx)^{n-1}y, (y^{-1}x^{-1})^{n}y(xy)^{n} \right)$ and
\item $\sigma^{2n+1}(x,y) = \left((y^{-1}x^{-1})^{n}y(xy)^{n},y^{-1}(x^{-1}y^{-1})^{n}x(yx)^{n}y\right)$.
\end{enumerate}
\end{lemma}

\begin{proof}

The proof of (a) is by induction on $n$. When $n=0$,
\[
\left(y^{-1}(x^{-1}y^{-1})^{-1}x(yx)^{-1}y, (y^{-1}x^{-1})^{0}y(xy)^{0}\right)
=\left(y^{-1}(yx)x(x^{-1}y^{-1})y, y\right)=(x, y)
\]
as required. Now assume the lemma holds when $n=k$, and let $(a,b) = \sigma^{2k}(x,y)$. By assumption, $b = (y^{-1}x^{-1})^{k}y(xy)^{k}$. Perform two Hurwitz moves on $(a,b)$ to obtain
\begin{equation*}
\sigma^{2(k+1)}(x,y) = (b^{-1}ab,b^{-1}a^{-1}bab).
\end{equation*}
Since Hurwitz moves on a factorization preserve its product, $ab = xy$ and $b^{-1}a^{-1} = y^{-1}x^{-1}$. By assumption $b = (y^{-1}x^{-1})^{k}y(xy)^{k}$. Thus
\begin{align*}
\sigma^{2(k+1)}(x,y) &= (b^{-1}xy,y^{-1}x^{-1}bxy)\\
&= \left((y^{-1}x^{-1})^{k}y^{-1}(xy)^{k}xy, y^{-1}x^{-1}(y^{-1}x^{-1})^{k}y(xy)^{k}xy\right) \\
&=\left(y^{-1}(x^{-1}y^{-1})^{(k+1)-1}x(yx)^{(k+1)-1}y,(y^{-1}x^{-1})^{k+1}y(xy)^{k+1} \right),
\end{align*}
and the induction is complete.\par
Now let $(a,b) = \sigma^{2n+1} (x,y)$. By part (a), $a = (y^{-1}x^{-1})^{n}y(xy)^{n}$. Since $ab = xy$, $b = (y^{-1}x^{-1})^{n}y^{-1}(xy)^{n}xy$. Thus 
\[
(a,b) = \left((y^{-1}x^{-1})^{n}y(xy)^{n},y^{-1}(x^{-1}y^{-1})^{n}x(yx)^{n}y\right).
\]

The proof for negative $n$ is similar.
\end{proof}

\begin{theorem}
\label{thm:xy-yx}
 Let $G$ be any group, and $x, y$ any elements of $G$.  The Hurwitz orbit of the factorization $(x, y)$ has the same size as the Hurwitz orbit of $(y, x)$.

\end{theorem}

\begin{proof}
Let $n$ be a nonnegative integer, and suppose $\sigma^{2n}(x,y) = (x,y)$. By Lemma \ref{lemma:xy-form}, 
\[
(x,y) = \left(y^{-1}(x^{-1}y^{-1})^{n-1}x(yx)^{n-1}y, (y^{-1}x^{-1})^{n}y(xy)^{n} \right).
\]
Observe that the right-hand value of this factorization may be written as $y^{-1}(x^{-1}y^{-1})^{n-1}x^{-1}(yx)^ny$, so that 
\begin{equation*}
\label{eqn:right-left}
y = (x^{-1}y^{-1})^{n-1}x^{-1}(yx)^n=x^{-1}(y^{-1}x^{-1})^{n-1}y(xy)^{n-1}x. 
\end{equation*}
By Lemma \ref{lemma:xy-form},
\begin{align*}
\sigma^{2n}(y,x) &= \left( x^{-1}(y^{-1}x^{-1})^{n-1}y(xy)^{n-1}x, (x^{-1}y^{-1})^n x(yx)^n\right)\\
&= \left(y, (x^{-1}y^{-1})^n x(yx)^n \right).
\end{align*}
The product of this factorization is $yx$, so the right-hand value is $x$, and $\sigma^{2n}(y,x) = (y,x)$. 

If $\sigma^{2n+1}(x,y) = (x,y)$, then by Lemma \ref{lemma:xy-form} 
\begin{equation}\label{eqn:xy-odd-left} 
y = y^{-1}(x^{-1}y^{-1})^nx(yx)^{n}y = (x^{-1}y^{-1})^nx(yx)^n. \end{equation}
By Lemma \ref{lemma:xy-form} and (\ref{eqn:xy-odd-left})
\begin{align*}
\sigma^{2n+1}(y,x) &= \left((x^{-1}y^{-1})^nx(yx)^n, x^{-1}(y^{-1}x^{-1})^ny(xy)^{n}x\right) \\
&= \left(y, x^{-1}(y^{-1}x^{-1})^ny(xy)^{n}x\right)
\end{align*}
So $\sigma^{2n+1}(y,x) = (y,x)$ \par
We have shown that for any nonnegative integer $n$, $\sigma^{n}(x,y) = (x,y)$ implies $\sigma^n(y,x) = (y,x)$. Thus the size of the Hurwitz orbit of $(y,x)$ is less than or equal to the size of the Hurwitz orbit of $(x,y)$. By interchanging $x$ and $y$ in the argument, it can be seen that $\sigma^n(y,x) = (y,x)$ implies $\sigma^n(x,y) = (x,y)$. So $(x,y)$ and $(y,x)$ have Hurwitz orbits of the same size.
\end{proof}

We expand on Lemma \ref{lemma:xy-form} to prove $(x, y)$ has the same sized Hurwitz orbit as $(x^{-1}, y^{-1})$.

\begin{theorem}
Let $G$ be any group, and $x, y$ any elements of $G$.  The Hurwitz orbit of the factorization $(x, y)$ has the same size as the Hurwitz orbit of $(x^{-1}, y^{-1})$.
\end{theorem}

\begin{proof}

Let $n$ be a nonnegative integer, and assume $\sigma^{2n}(x,y) = (x,y)$. By Lemma \ref{lemma:xy-form}, $y = (y^{-1}x^{-1})^ny(xy)^n$, so $(xy)^ny =  y(xy)^n $. Conjugate on both sides by $y^{-1}$ to obtain
\begin{align*}
y(xy)^nyy^{-1} &= yy(xy)^ny^{-1} \\
(yx)^ny&= y(yx)^n.
\end{align*}
Multiplying this last equality on both sides by $y^{-1}$, we have that $y^{-1}$ and $(yx)^n$ commute. So $y^{-1} = (yx)^{n}y^{-1}(x^{-1}y^{-1})^{n}$, which is precisely the right-hand value of $\sigma^{2n}(x^{-1},y^{-1})$ by Lemma \ref{lemma:xy-form}. Thus $\sigma^{2n}(x^{-1},y^{-1}) = (x^{-1},y^{-1})$. \par
If $\sigma^{2n+1}(x,y) = (x,y)$, then Lemma \ref{lemma:xy-form} gives
\begin{equation*}
y = y^{-1}(x^{-1}y^{-1})^nx(yx)^ny,
\end{equation*}
so that $y^{-1} = y^{-1}(x^{-1}y^{-1})^{n}x^{-1}(yx)^{n}y=(x^{-1}y^{-1})^{n}x^{-1}(yx)^{n}$, or equivalently $x^{-1} = (yx)^ny^{-1}(x^{-1}y^{-1})^n$. But, by Lemma \ref{lemma:xy-form}, this implies the left-hand value of $\sigma^{2n+1}(x^{-1},y^{-1})$ is $x^{-1}$, and so necessarily the right-hand value must be $y^{-1}$.\par
Altogether we have shown that $\sigma^n(x,y)=(x,y)$ implies $\sigma^n(x^{-1},y^{-1}) = (x^{-1},y^{-1})$. With little modification the same argument can be used to show that the converse is true as well. Thus $(x,y)$ and $(x^{-1},y^{-1})$ have Hurwitz orbits of the same size.
\end{proof}

\section{Cycling}
\label{section:cycling}

In this section, we show that cycling elements in a factorization leaves Hurwitz orbit size unchanged.

\begin{lemma}
\label{lem:conjugation}
Let $G$ be any group containing elements $x_1, x_2, \ldots , x_\ell$ and $y$. Then $(x_{1},x_{2},\ldots,x_{\ell})$ and $(y^{-1}x_{1}y,y^{-1}x_{2}y,\ldots,y^{-1}x_{\ell}y)$ have Hurwitz orbits of the same size.
\end{lemma}

\begin{proof}

To prove that the orbits are of equal size, we will find a bijection between them. First we show that if 
\begin{equation*}
\sigma_{p_1} \ldots \sigma_{p_n}(x_{1},x_{2},\ldots,x_{\ell})=(z_{1},z_{2},\ldots,z_{\ell}),
\end{equation*}
then
\begin{equation*}
\sigma_{p_1} \ldots \sigma_{p_n}(y^{-1}x_{1}y,y^{-1}x_{2}y,\ldots,y^{-1}x_{\ell}y)=(y^{-1}z_{1}y,y^{-1}z_{2}y,\ldots,y^{-1}z_{\ell}y),
\end{equation*}
where each $p_i \in \{1, \ldots, \ell - 1\}$. We verify that this is true when one Hurwitz move is applied: if $1 \leq i \leq \ell - 1$, then 
\begin{align*}
(\ldots, x_i,x_{i+1}, \ldots) &\overset{\sigma_i}{\to} (\ldots, x_{i+1},x_{i+1}^{-1}x_i x_{i+1}, \ldots), \\
(\ldots,y^{-1}x_{i}y,y^{-1}x_{i+1}y, \ldots) &\overset{\sigma_i}{\to} (\ldots, y^{-1}x_{i+1}y,y^{-1}x_{i+1}^{-1}x_{i}x_{i+1}y, \ldots). 
\end{align*}
The result follows by induction.

Observe that this if/then relationship goes both ways: $(y^{-1}x_{1}y,y^{-1}x_{2}y,\ldots,y^{-1}x_{\ell}y)$ is the factorization obtained by conjugating every element in $(x_{1},x_{2},\ldots,x_{\ell})$ by $y$, and $(x_{1},x_{2},\ldots,x_{\ell})$ is the factorization obtained by conjugating every element in $(y^{-1}x_{1}y,y^{-1}x_{2}y,\ldots,y^{-1}x_{\ell}y)$ by $y^{-1}$ (if $y$ is in $G$, so is $y^{-1}$).

We may define a function mapping each factorization $(z_{1},z_{2},\ldots,z_{\ell})$ in the Hurwitz orbit of $(x_{1},x_{2},\ldots,x_{\ell})$ to $(y^{-1}z_{1}y,y^{-1}z_{2}y,\ldots,y^{-1}z_{\ell}y)$ in the orbit of $(y^{-1}x_{1}y,y^{-1}x_{2}y,\ldots,y^{-1}x_{\ell}y)$. This function is the desired bijection.
\end{proof}

\begin{theorem}
\label{thm:cycling}
Let $G$ be any group containing elements $x_1, x_2, \ldots , x_\ell$. Then $(x_{1},x_{2},\ldots,x_{\ell})$ and $(x_{2},\ldots,x_{\ell},x_{1})$ have Hurwitz orbits of the same size.
\end{theorem}

\begin{proof}
Suppose we apply the Hurwitz moves $\sigma_{\ell - 1}$, $\sigma_{\ell - 2}$, \ldots, $\sigma_1$ in that order to the factorization $(x_{2},\ldots,x_{\ell},x_{1})$. The resulting factorization is $(x_{1},x_{1}^{-1}x_{2}x_{1},\ldots,x_{1}^{-1}x_{\ell}x_{1})$. This resulting factorization is identical to $(x_{1},x_{2},\ldots,x_{\ell})$, except every element has been conjugated by $x_{1}$; in particular $x_{1}=x_{1}^{-1}x_{1}x_{1}$. By Lemma~\ref{lem:conjugation}, our theorem is true.
\end{proof}

\section{Reverses}

In this section we show that $(x_{1},x_{2}, \ldots, x_{\ell})$ will always have the same sized Hurwitz orbit as $(x_{\ell}^{-1}, \ldots,x_{2}^{-1},x_{1}^{-1})$. We also expand on some ideas presented in Section \ref{section:cycling}.

\label{reverses}
\begin{theorem}
\label{thm:flip}
Let $G$ be any group containing elements $x_1, x_2, \ldots , x_\ell$. Then $(x_{1},x_{2}, \ldots, x_{\ell})$ and $(x_{\ell}^{-1}, \ldots,x_{2}^{-1},x_{1}^{-1})$ have Hurwitz orbits of the same size.
\end{theorem}

\begin{proof}
To prove that the orbits are of equal size, we will find a bijection between them. First we show that if
\begin{equation*}
\sigma_{p_1} \ldots \sigma_{p_n}(x_{1},x_{2},\ldots,x_{\ell})=(z_{1},z_{2},\ldots,z_{\ell}),
\end{equation*}
then
\begin{equation*}
\label{eqn:conj2}
\sigma_{\ell-p_1}^{-1} \ldots \sigma_{\ell-p_n}^{-1}(x_{\ell}^{-1}, \ldots, x_2^{-1}, x_1^{-1})=(z_{\ell}^{-1}, \ldots, z_2^{-1}, z_1^{-1})
\end{equation*}
where each $p_i \in \{1, \ldots, \ell - 1\}$. We verify that this is true when one Hurwitz move is applied: if $1 \leq i \leq \ell - 1$, then
\begin{align*}
(\ldots, x_i,x_{i+1}, \ldots) &\overset{\sigma_i}{\to} (\ldots, x_{i+1},x_{i+1}^{-1}x_i x_{i+1}, \ldots), \\
(\ldots,x_{i+1}^{-1},x_{i}^{-1}, \ldots) &\overset{\sigma_{\ell-i}^{-1}}{\to} (\ldots, x_{i+1}^{-1}x_{i}^{-1}x_{i+1},x_{i+1}^{-1}, \ldots). 
\end{align*}
The result follows by induction. This if/then relationship goes both ways because $(x_k^{-1})^{-1}=x_k$.

We may define the following function: some factorization $(z_{1},z_{2},\ldots,z_{\ell})$ in the Hurwitz orbit of $(x_{1},x_{2},\ldots,x_{\ell})$ is mapped to the factorization $(z_{\ell}^{-1}, \ldots, z_2^{-1}, z_1^{-1})$, which is in the Hurwitz orbit of $(x_{\ell}^{-1}, \ldots, x_2^{-1}, x_1^{-1})$. This function is the desired bijection.
\end{proof}

One interesting case of Theorem \ref{thm:flip} is when every element in a factorization has order $1$ or $2$.

\begin{corollary}
\label{reverse}
Let $G$ by any group containing $x_1, x_2 \ldots, x_\ell$ such that $x_k$ has order $1$ or $2$ for $k = 1, \ldots, \ell$. Then $(x_{1},x_{2}, \ldots ,x_{\ell})$ and $(x_{\ell}, \ldots,x_{2},x_{1})$ have Hurwitz orbits of the same size.
\end{corollary}

\begin{proof}
If $x_{k}$ has order $1$ or $2$, then $x_{k}=x_{k}^{-1}$. So the factorization $(x_{\ell}^{-1}, \ldots,x_{2}^{-1},x_{1}^{-1})$ is equal to $(x_{\ell}, \ldots,x_{2},x_{1})$. Therefore, by Theorem~\ref{thm:flip} this corollary is true.
\end{proof}

\begin{remark}
In Theorem ~\ref{thm:cycling} we prove that the factorizations $(x_{1}, x_2, \ldots, x_\ell)$ and $(x_{2},\ldots, x_\ell, x_1)$ will always have the same sized Hurwitz orbits. This theorem produces particularly interesting results when observing factorizations containing only two distinct elements.

Suppose we have a factorization $(x_{1}, x_2, \ldots, x_\ell)$ such that each $x_i$, $1 \leq i \leq \ell$, is equal to $y$ or $z$. One can show that if the factorization has length $\ell$, $\ell \leq 5$, then $(x_{1}, x_2, \ldots, x_\ell)$ has the same sized Hurwitz orbit as its reverse, $(x_{\ell},\ldots, x_2, x_1)$. This is because $(x_{\ell},\ldots, x_2, x_1)$ can be obtained by cycling $(x_{1}, x_2, \ldots, x_\ell)$.

The following is an example. Suppose we want to show that $(z,z,y,z,y)$ has the same sized Hurwitz orbit as $(y,z,y,z,z)$. Using Theorem ~\ref{thm:cycling}, we can show that $(z,z,y,z,y)$ has the same sized Hurwitz orbit as $(z,y,z,y,z)$ which has the same sized Hurwitz orbit as $(y,z,y,z,z)$.

However, this property does not neccesarily hold when there are more than two distinct elements in the factorization, or when the length of the factorization is greater than 5. For example, suppose you take the factorization $(x,y,z)$ and you want to prove that its Hurwitz orbit has the same size as its reverse, $(z,y,x)$. Cycling, we get $(x,y,z) \approx (y,z,x) \approx (z,x,y) \approx (x,y,z)$. We have cycled completely through $(x,y,z)$ without producing $(z,y,x)$. Now let's take a look at a case where the factorization has length greater than 5. Let's take the factorization $(x,x,y,x,y,y)$ and show that we cannot produce its reverse $(y,y,x,y,x,x)$ through cycling:
\begin{multline*}
(x,x,y,x,y,y) \approx (x,y,x,y,y,x) \approx (y,x,y,y,x,x) \approx 
\\ 
(x,y,y,x,x,,y) \approx (y,y,x,x,y,x) \approx (y,x,x,y,x,y) \approx (x,x,y,x,y,y).
\end{multline*}
We have shown a case of a length 6 factorization where cycling does not produce its reverse. Therefore, our Remark proves that a factorization necessarily, by Theorem ~\ref{thm:cycling}, has the same sized Hurwitz orbit as it's reverse if and only if the length $\ell$ of the factorization is less than or equal to 5, and there are only 2 distinct elements in the factorization. 
\end{remark}

\section{Double reverses and reverse relations}
\label{Section:ReverseRelation}

In this section, we investigate groups with presentations in which relations between the generators are reversible. We prove that in such groups, Hurwitz orbit size is preserved when reversing factorizations of elements from the generating set. This result follows as a corollary of the main theorem in this section, Theorem \ref{dbr-hsize}, which proves that a modified reversal operation on factorizations consisting of arbitrary words preserves Hurwitz orbit size.

We begin by recalling a number of standard definitions from the literature -- see, for example, \cite{Peifer}. Let $X = \{x_1, x_2, \ldots, x_n\}$ be a set of distinct elements, and $X^{-1} = \{x_1^{-1}, x_2^{-1}, \ldots, x_n^{-1}\}$ a set of elements distinct from each other and from the elements of $X$. A \emph{word} on $X \cup X^{-1}$ is a string of finitely many elements, or \emph{letters}, from $X \cup X^{-1}$. An \emph{inverse pair} is a word of the form $x_ix_i^{-1}$ or $x_i^{-1}x_i$, and a word is said to be \emph{reduced} if it contains no inverse pairs. The set $F(X)$ of all reduced words on $X \cup X^{-1}$ is a group under the operation of concatenation followed by deletion of inverse pairs; it is called the \emph{free group on $n$ generators}. The inverse of an element $x_{i_1} x_{i_2} \ldots x_{i_p} \in F(X)$ is $x_{i_p}^{-1}x_{i_{p-1}}^{-1} \ldots x_{i_1}^{-1}$.

Now let $G$ be a group and $S$ a subset of $G$. The \emph{normal closure} of $S$ in $G$ is defined as the intersection of all normal subgroups of $G$ containing $S$. Every element of the normal closure of $S$ in $G$ is the product of conjugates of elements of $S$.

A \emph{presentation} $\langle X \mid R \rangle$ of a finitely generated group $G$ consists of a generating set $X = \{x_1, x_2, \ldots, x_n\}$ and a set of \emph{relations} $R \subset F(X)$, such that the quotient of $F(X)$ by the normal closure of $R$ in $F(X)$ is isomorphic to $G$. We denote the normal closure of $R$ in $F(X)$ by $N$.

The set $X$ is identified with a generating set for $G$, and two elements $x_{i_1}x_{i_2} \ldots x_{i_p}$ and $x_{j_1} x_{j_2} \ldots x_{j_q}$ are equal in $G$ if and only if they belong to the same coset of $N$ in $F(X)$. To distinguish between the two groups, it will be useful to write $a = b$ for equality in $F(X)$ and $a \equiv b$ for equality in $G$. 

If $a \in F(X)$, the \emph{reverse} $a^*$ of $a$ is obtained by reversing the order of the letters in $a$. If $a = x_{i_1} x_{i_2} \ldots x_{i_p}$, then $a^* = x_{i_p}x_{i_{p-1}} \ldots x_{i_1}$. If some inverse pair $x_ix_i^{-1}$ appears in $a^*$, then $x_i^{-1}x_i$ appears in $a$, a contradiction because $a$ is reduced. So the reverse of a reduced word is itself a reduced word. With this notation we have 
\begin{equation}
\label{eqn:rev-inv}
\left(a^{-1}\right)^* = x_{i_1}^{-1} x_{i_2}^{-1} \ldots x_{i_p}^{-1} = \left(a^*\right)^{-1}. 
\end{equation}

We begin with a proposition on how to calculate reverses of products in $F(X)$.

\begin{prop}
\label{prop:fn-rev-mult} 
Let $a, a_1, a_2, \ldots, a_m$ be elements of $F(X)$ such that $a = a_1 a_2 \ldots a_m$. Then $a^* = a_m^* a_{m-1}^* \ldots a_1^*$. That is, multiplying words in $F(X)$ and reversing their product is equivalent to reversing each word, reversing the order of multiplication, and taking the product. 
\end{prop}
\begin{proof}
It suffices to verify the statement when $a$ is a product of two elements $a_1$ and $a_2$ of $F(X)$. The proposition then follows by induction. 

Suppose $a = x_{i_1}x_{i_2} \ldots x_{i_p}$, $a_1 = x_{j_1} x_{j_2} \ldots x_{j_q}$, and $a_2 = x_{k_1} x_{k_2} \ldots x_{k_r}$. We want to show $x_{i_p}x_{i_{p-1}} \ldots x_{i_1} = (x_{k_r} x_{k_{r-1}} \ldots x_{k_1})(x_{j_q} x_{j_{q-1}} \ldots x_{j_1})$. Since $a_1$ and $a_2$ are reduced, either $x_{k_r}x_{k_{r-1}} \ldots x_{k_1} x_{j_q}x_{j_{q-1}} \ldots x_{j_1}$ is reduced or $x_{k_1}x_{j_q}$ is an inverse pair. In the first case it follows that $x_{j_1} x_{j_2} \ldots x_{j_q} x_{k_1} x_{k_2} \ldots x_{k_r}$ is reduced and is thus equal to $a$. So $a^* = x_{k_r}x_{k_{r-1}} \ldots x_{k_1} x_{j_q}x_{j_{q-1}}x_{j_1} = a_2^*a_1^*$, as required.

If $x_{k_1}x_{j_q}$ is an inverse pair, then so is its reverse. Deleting these inverse pairs gives
\begin{align*} a &= (x_{j_1} x_{j_2} \ldots x_{j_{q-1}})(x_{k_2} x_{k_3} \ldots x_{k_r}), \\
a_2^*a_1^* &= (x_{k_r}x_{k_{r-1}} \ldots x_{k_2}) (x_{j_{q-1}}x_{j_{q-2}} \ldots x_{j_1}).
\end{align*}
This process continues until there are no more inverse pairs to be deleted.
\end{proof}

The groups of interest in this section can now be defined. A presentation $\langle X \mid R \rangle$ for a group $G$ is said to have \emph{reversible relations} if for every $x_{i_1} x_{i_2} \ldots x_{i_p} \in R$ the reverse word $x_{i_p} x_{i_{p-1}} \ldots x_{i_1}$ is contained in $N$.

For example, two commonly used presentations of the dihedral group of order $2n$ have reversible relations. Let $D_{2n} = \langle r,s \mid r^n = s^2 = 1, rs = sr^{-1} \rangle$. Formally we can say $D_{2n} = \langle r,s \mid R \rangle$ where $R = \{r^n, s^2, rsrs^{-1}\}$. Then $(r^n)^* = r^n \in R$, $(s^2)^* = s^2 \in R$, and $(rsrs^{-1})^* = s^{-1}rsr = s^{-1}(rsrs^{-1})s \in N$, so this presentation has reversible relations.

Now consider the presentation $\langle a,b \mid a^2 = b^2 = (ab)^n = 1 \rangle$ for $D_{2n}$. Since the reverse of $(ab)^n$ is $(ba)^n = b(ab)^nb^{-1} \in N$, this presentation again has reversible relations. An important part of our definition of reversible relations is that reverse words are not required to be in the set of relations $R$ itself, but rather in the normal closure $N$ of $R$.

The presentation $\langle a,b \mid a^4 = 1, a^2 = b^2, ba = a^{-1}b \rangle$ for the quarternion group $Q_8$ is another example of a presentation with reversible relations. However, $Q_8$ can also be given by $\langle -1, i, j, k \mid (-1)^2 = 1, i^2 = j^2 = k^2 = ijk = -1 \rangle$, which does not have the same property: while $ijk(-1) = 1$, the reverse word $(-1)kji$ is equal to $(-1)k(-1)k = k^2 = -1 \neq 1$.

\begin{remark}
\label{remark:presentations}
Suppose $\langle X \mid R \rangle$ is a group presentation, and we want to show that it has reversible relations. Our definition would have us reverse each word in $R$ and check that these reverses belong to $N$. This process can be made easier by the following observations.
Relations giving the order of a generator, such as $r^n = 1$, are always reversible. Equations of the form $a = a^*$, such as $xyxy = yxyx$, are reversible. Equations $a = b$ where $a$ and $b$ are \emph{palindromes}, which is to say $a = a^*$ and $b = b^*$ in $F(X)$, are reversible. 
\end{remark} 

With an arbitrary presentation, the reverses of two equal words need not be equal. For example, in $Q_8$ we have $ij = k$ but $ji = -k \neq k$. The next proposition shows that this never happens when reversing words on a generating set with reversible relations. 

\begin{prop}
\label{rev-eq}
Suppose that the group $G$ is given by a presentation $\langle X \mid R \rangle$ with reversible relations. If $a,b \in F(X)$ and $a \equiv b$ in $G$, then $a^* \equiv b^*$.
\end{prop}
\begin{proof}
The elements $a$ and $b$ belong to the same coset of $N$ in $F(X)$, so $a = bn$ for some $n \in N$. Proposition \ref{prop:fn-rev-mult} gives $a^* = n^* b^*$, so $a^*$ and $b^*$ belong to the same coset of $N$ if $n^* \in N$. That is, $a^* \equiv b^*$ if $N$ is closed under the operation of reversing words.

Every element $n \in N$ is the product of conjugates of elements of $R$. If 
\[ n = b_1a_1b_1^{-1}b_2a_2b_2^{-1} \cdots b_ma_mb_m^{-1} \]
with $a_i \in R$, $b_i \in F(X)$, $i = 1, 2, \ldots, m$, then by Proposition \ref{prop:fn-rev-mult}
\[ n^* = \left(b_m^{-1}\right)^* a_m^* b_m^* \left(b_{m-1}^{-1}\right)^* a_{m-1}^* b_{m-1}^*  \cdots \left(b_1^{-1}\right)^* a_1^* b_1^*.\]
The presentation $\langle X \mid R \rangle$ has reversible relations, so $a_i^* \in N$ for all $i$. Since $(b_i^{-1})^* = (b_i^*)^{-1}$, $n^*$ is the product of conjugates of elements of $N$, and thus $n^* \in N$.
\end{proof} 

The main theorem in this section concerns an operation on factorizations which does more than just reverse the order of the factors: if $\langle X \mid R \rangle$ is a presentation for the group $G$, and if $U$ and $V$ are length-$\ell$ factorizations of elements in $G$, we say $U$ is a \emph{double reverse} of $V$ if there exist $a_1, a_2, \ldots, a_\ell \in F(X)$ such that 
\begin{align*}
U &= \left( a_1, a_2, \ldots, a_\ell \right), \\
V &= \left( a_\ell^*, a_{\ell-1}^*, \ldots, a_1^* \right).
\end{align*}
That is, $U$ is a double reverse of $V$ if $V$ can be obtained by reversing the order in which words appear in $U$ and reversing each word. This relation is symmetric, so we say $U$ and $V$ are double reverses.

\begin{remark}
Double reverses are not unique. For example, with the presentation for $Q_8$ given above we have that $T = (i, j, ij)$ and $U = (ji, j, i)$ are double reverses. Also $V = (k, j, i)$ is a double reverse of $T$ because $ij = k$. But $ji = -k \neq k$, so $U \neq V$.
\end{remark}

The next lemma lets us preserve the double reverse relation after performing a series of Hurwitz moves on one factorization. We will use it to build a map between the Hurwitz orbits of double reverse factorizations.

\begin{lemma}
\label{dbr-moves}
Let $G = \langle X \mid R \rangle$, let $U$ and $V$ be length-$\ell$ factorizations of elements in $G$, and assume $U$ and $V$ are double reverses. Perform $n$ Hurwitz moves $\sigma_{p_1}, \sigma_{p_2}, \ldots, \sigma_{p_n}$ on $U$ in order as listed, with $p_i \in \{1, \ldots, \ell-1\}$ for $i = 1,\ldots,n$. Call the factorization produced by these moves $U_n$. Let $V_n$ be the factorization obtained by performing inverse Hurwitz moves $\sigma^{-1}_{\ell - p_1}, \sigma^{-1}_{\ell - p_2}, \ldots, \sigma^{-1}_{\ell - p_n}$ on $V$. Then $U_n$ and $V_n$ are double reverses.
\end{lemma}

\begin{proof}
Induct on $n$; the base case $n=0$ follows from having defined $U$ and $V$ as double reverses. Let $k$ be a nonnegative integer such that the lemma holds when $n=k$, and choose integers $p_1, \ldots, p_{k+1}$ between $1$ and $\ell-1$. Let 
\begin{align*}
U_k &= \sigma_{p_k} \sigma_{p_{k-1}} \cdots \sigma_{p_1}\left(U\right),\\
V_k &= \sigma^{-1}_{\ell-p_k} \sigma^{-1}_{\ell-p_{k-1}} \cdots \sigma^{-1}_{\ell-p_1}\left(V\right).
\end{align*}
By assumption $U_k$ and $V_k$ are double reverses, so there exist $a_1, a_2, \ldots, a_\ell \in F(X)$ such that
\begin{align*}
U_k &= \left( a_1, a_2, \ldots, a_\ell \right), \\
V_k &= \left( a_\ell^*, a_{\ell-1}^*, \ldots, a_1^* \right).
\end{align*} \par
Now let $p = p_{k+1}$, $U_{k+1} = \sigma_{p}(U_k)$, and $V_{k+1} = \sigma^{-1}_{\ell-p}(V_k)$. The Hurwitz move $\sigma_p$ only modifies elements at positions $p$ and $p+1$ in $U_k$, and $\sigma^{-1}_{\ell-p}$ only modifies elements at positions $\ell-p$ and $\ell-p+1$ in $V_k$:
\begin{align*}
U_{k+1} &= ( a_1, \ldots, a_{p-1}, \underbrace{a_{p+1}}_{p}, \underbrace{a_{p+1}^{-1}a_{p}a_{p+1}}_{p+1}, a_{p+2}, \ldots, a_\ell ),\\
V_{k+1} &= (a_\ell^*, \ldots, a_{p+2}^*, \underbrace{a_{p+1}^*a_{p}^* \left(a_{p+1}^*\right)^{-1}}_{\ell-p}, \underbrace{a_{p+1}^*}_{\ell-p+1}, a_{p-1}^*, \ldots a_1^* ).
\end{align*}
Reversing the order of elements in $U_{k+1}$ and changing each $a_i$ to $a_i^*$ produces a factorization identical to $V_{k+1}$ except at position $\ell - p$, so the two factorizations are double reverses if factor $\ell - p$ in $V_{k+1}$ is the reverse of factor $p+1$ in $U_{k+1}$. But \[(a_{p+1}^{-1}a_{p}a_{p+1})^* = a_{p+1}^*a_{p}^* \left(a_{p+1}^{-1}\right)^* = a_{p+1}^*a_{p}^* \left(a_{p+1}^*\right)^{-1}\] by \eqref{eqn:rev-inv}  and Proposition \ref{prop:fn-rev-mult}, so $U_{k+1}$ and $V_{k+1}$ are indeed double reverses.
\end{proof}

\begin{theorem}
\label{dbr-hsize}
Let $G$ be a group given by a presentation $\langle X \mid R \rangle$ with reversible relations. Let $U$ be a length-$\ell$ factorization of elements in $G$, and let $V$ be a double reverse of $U$. Then $U$ and $V$ have Hurwitz orbits of the same size.
\end{theorem}
\begin{proof}
Let $O_U$ and $O_V$ be the orbits of $U$ and $V$, respectively. Suppose $T \in O_U$ and $T = \sigma_{p_1} \sigma_{p_2} \ldots \sigma_{p_n} (U)$ with each $p_i \in \{1, 2, \ldots, \ell - 1\}$, and define $\varphi: O_U \rightarrow O_V$ by $\varphi(T) = \sigma^{-1}_{\ell - p_1} \sigma^{-1}_{\ell - p_2} \ldots \sigma^{-1}_{\ell - p_n}(V)$. By Lemma \ref{dbr-moves}, $T$ and $\varphi(T)$ are double reverses.

Assume $A,B \in O_U$, $C \in O_V$, and $\varphi(A) = \varphi(B) = C$. Then there exist  $a_1, a_2, \ldots, a_\ell, b_1, b_2, \ldots, b_\ell \in F_n$ such that
\[ A = \left( a_1, a_2, \ldots, a_\ell \right) \quad \text{and} \quad B = \left( b_1, b_2, \ldots, b_\ell \right), \]
\[ C = \left( a_\ell^*, a_{\ell-1}^*, \ldots, a_1^* \right) = \left( b_\ell^*, b_{\ell-1}^*, \ldots, b_1^* \right). \]
Since $a_i^* \equiv b_i^*$, by Proposition \ref{rev-eq} $a_i \equiv b_i$, $i = 1, \ldots, \ell$. Thus $A = B$, $\varphi$ is injective, and the same process gives an injective map $\varphi': O_V \rightarrow O_U$ as well. This proves $\left|O_U\right| = \left|O_V\right|$.
\end{proof}

\begin{corollary}
\label{dbr-cor}
Suppose $\langle X \mid R \rangle$ is a presentation for $G$ with reversible relations, and let $x_{i_1}, x_{i_2}, \ldots, x_{i_\ell} \in X \cup X^{-1}$. Then $(x_{i_1}, x_{i_2}, \ldots, x_{i_\ell})$ and $(x_{i_\ell}, x_{i_{\ell -1}}, \ldots, x_{i_1})$ have Hurwitz orbits of the same size.
\end{corollary}
\begin{proof}
The two factorizations are double reverses because the reverse of a single letter word is itself.
\end{proof}

\section{Applications to reflection groups}
\label{Section:Applications}

A lot of the research done on the Hurwitz action has been on real reflection groups, and more generally on Coxeter groups. The reflections in these groups exclusively have order $2$. Therefore, by Corollary \ref{reverse}, we find that a reflection factorization $(x_1, x_2, \ldots , x_n)$ in a Coxeter group has the same sized Hurwitz orbit as $(x_n, \ldots , x_2, x_1)$.

The Hurwitz action has also been studied in the context of complex reflection groups called \emph{Shephard groups}. Unlike Coxeter groups, Shephard groups can have generating reflections of order greater than 2. Thus Corollary \ref{reverse} is not enough to guarantee that reversing a factorization of generators preserves Hurwitz orbit size. For each Shephard group there exist integers $p_1, p_2, \ldots, p_n$ and $q_1, q_2, \ldots, q_{n-1}$ such that there is a generating set $\{s_1, s_2, \ldots, s_n\}$ with relations
$s_i^{p_i} = 1$ for $i = 1, 2, \ldots, n$, $s_i s_j = s_j s_i$ if $\left|i - j\right| > 1$, and \[s_is_{i+1}s_is_{i+1} \ldots = s_{i+1}s_is_{i+1}s_i \ldots \quad (q_i \text{ terms on each side}) \] for $i = 1,2,\ldots,n-1$. We see with Remark \ref{remark:presentations} that these relations are reversible, so by Corollary \ref{dbr-cor} reversing any factorization consisting of generators $s_i$ in a Shephard group preserves Hurwitz orbit size. For example, the group $G_6$ mentioned in the introduction is a Shephard group, with parameters $n = 2$, $p_1 = 3$, $p_2 = 2$, and $q_1 = 6$. Yet testing of small cases has led us to make a stronger conjecture about Hurwitz orbit sizes in $G_6$. 

\begin{conjecture}
\label{conjecture:g6}
Let $G_6 = \langle a,b \mid a^3 = b^2 = 1, ababab = bababa \rangle$, and suppose $x_1, x_2, \ldots, x_\ell \in \{a,b, a^{-1}\}$ for some positive integer $\ell$. Then for any permutation $\pi$ of $\{1,2, \ldots, \ell \}$, $(x_1, x_2, \ldots, x_\ell)$ and $(x_{\pi(1)}, x_{\pi(2)}, \ldots, x_{\pi(\ell)})$ have Hurwitz orbits of equal size.
\end{conjecture}

With regard to this property, $G_6$ may be unique even among Shephard groups. For example, in $G_4 = \langle a,b \mid a^3 = b^3 = 1, aba = bab \rangle$, $(a,a,b,b)$ has a Hurwitz orbit of size $36$ while $(a,b,a,b)$ has an orbit of size $27$.

\end{document}